\documentclass[onefignum,onetabnum]{siamart220329}
\usepackage{amsmath}
\usepackage{amssymb}
\usepackage{algorithm}
\usepackage{algorithmic}
\usepackage{subfig}
\usepackage{graphicx}
\usepackage{mdframed}
\usepackage{mathtools}
\usepackage{float}

\newsiamremark{remark}{Remark}
\newsiamremark{hypothesis}{Hypothesis}
\crefname{hypothesis}{Hypothesis}{Hypotheses}
\newsiamthm{claim}{Claim}

\title{Unprojected recycled Block Krylov subspace methods for shifted systems
}
\author{Liam Burke\thanks{School Of Mathematics, Trinity College Dublin, College Green, Dublin 2, Ireland  (email: \email{burkel8@tcd.ie})}
}

\usepackage{amsopn}

\ifpdf
\hypersetup{
  pdftitle={Unprojected Recycled Block Krylov subspace methods for shifted systems},
  pdfauthor={Liam Burke}
}
\fi

\begin{document}

\maketitle

\begin{abstract}
The use of block Krylov subspace methods for computing the solution to a sequence of shifted linear systems using subspace recycling was first proposed in [Soodhalter, SISC 2016], where a recycled shifted block GMRES algorithm (rsbGMRES) was proposed. Such methods use the equivalence of the shifted system to a Sylvester equation and exploit the shift invariance of the block Krylov subspace generated from the Sylvester operator. This avoids the need for initial residuals to span the same subspace and allows for a viable restarted Krylov subspace method with recycling for solving sequences of shifted systems.

In this paper we propose to develop these types of methods using \textit{unprojected} Krylov subspaces.  In doing so we show how one can overcome the difficulties associated with developing methods based on projected Krylov subspaces such as rsbGMRES, while also allowing for practical methods to fit within a well known residual projection framework. In addition, unprojected methods are known to be advantageous when the projector is expensive to apply, making them of significant interest for High-Performance Computing applications.  We develop an unprojected rsbFOM and unprojected rsbGMRES. We also develop a procedure for extracting shift dependent harmonic Ritz vectors over an augmented block Krylov subspace for shifted systems yielding an approach for selecting a new recycling subspace after each cycle of the algorithm. Numerical experiments demonstrate the effectiveness of our methods.
\end{abstract}


\overfullrule=0pt 
\section{Introduction}
Many applications in High-Performance Computing (HPC) require one to solve a family of $N$ shifted linear systems of the form
\begin{equation}\label{eq:shiftedsystem}
    ( \textbf{A}^{(\ell)} + \sigma^{(\ell)}_{i} \textbf{I}) \textbf{x}^{(\ell)}(\sigma_{i}) = \textbf{b}^{(\ell)}(\sigma_{i}), \hspace{1cm} \ell = 1,..,N, \hspace{1cm} i = 1,2,..,s^{(\ell)},
\end{equation}
where system $\ell$ has coefficient matrix $\textbf{A}^{(\ell)} \in \mathbb{C}^{n \times n}$ and $s^{(\ell)}$ scalar shifts $\sigma_{i}^{(\ell)} \in \mathbb{C}$. Problems of this type arise in Lattice QCD simulations when using the Rational Hybrid Monte Carlo algorithm to generate QCD gauge fields \cite{birk2012cg, kennedy2008algorithms}. It also arises in the regularization of ill posed inverse problems \cite{MR1694686}. In the HPC setting the matrix $\textbf{A}^{(\ell)}$ is typically large and sparse and Krylov subspace methods are the only efficient means of computing an approximate solution for the shifted system. Krylov subspace methods work by projecting the problem onto a low dimensional Krylov subspace and solving a smaller representation of the problem on that subspace via a direct method. These methods are attractive because they are built only from matrix vector products, making them well suited when the matrix is not stored explicitly. In addition, Krylov subspaces are a natural choice of subspace in which a good approximate solution for every shift can be sought. Another advantage when solving shifted systems of the form (\ref{eq:shiftedsystem}) arises when the initial residuals for each system are linearly dependent and satisfy  $\textbf{r}_{0}^{(i)} = \beta \textbf{r}_{0}^{(j)}, \forall  i,j \in \{ 1,2,..,s \}, \beta \in \mathbb{C}$ where $r_{0}^{(i)}$ is the initial residual corresponding to shift $i$. In this case the following \textit{shift invariance} property is satisfied
\begin{equation}\label{eq:shiftinvariance}
\mathcal{K}_{j}(\textbf{A}^{(\ell)} + \sigma_{i} \textbf{I}, \textbf{r}_{0}^{(i)}) = \mathcal{K}_{j}(\textbf{A}^{(\ell)} + \sigma_{j} \textbf{I}, \textbf{r}_{0}^{(j)})
\end{equation}
with $\sigma_{i},\sigma_{j} \in \mathbb{C} \backslash\{ 0 \}$. 
This collinearity property allows for each shifted system to be solved using the same Krylov subspace and thus avoids extra storage or additional matrix vector products per shift. With initial collinear residuals a cycle of the Full Orthogonalization Method (FOM) \cite{MR616364} or the Generalized Minimum Residual Method (GMRES) \cite{MR848568} can be used to compute an approximate solution over the \textit{base} Krylov subspace $\mathcal{K}_{j}(\textbf{A},\textbf{r}_{0}^{(1)})$.

A cycle of the FOM algorithm always produces collinear residuals, and thus a restarted shifted FOM algorithm  can easily be implemented \cite{MR2011689}. After a cycle of the GMRES algorithm the residuals are in general no longer collinear and developing a restarted shifted GMRES algorithm requires one to either enforce the collinearity condition after each cycle ( see e.g. \cite{MR1616874}), or else work with methods which avoid the collinearity restriction completely. The same issues are encountered when developing augmented or recycled Krylov subspace methods for shifted systems \cite{MR3212178}. In an augmented Krylov subspace method the Krylov subspace is augmented with some fixed $k$ dimensional subspace $\mathcal{U}$ in order to accelerate convergence of the iterative solver. In a \textit{recycled} Krylov subspace method this augmentation subspace is constructed from a Krylov subspace iteration for solving a previous system in a sequence \cite{MR3054589, MR3284571,MR2272183, MR4175151}. The recycled FOM method no longer has naturally collinear residuals after restart and thus also encounters difficulties with the collinearity restriction. Further, a recycled GMRES algorithm continues to suffer this restriction.

In \cite{MR3530183}, a recycled shifted block GMRES (rsbGMRES) method was proposed to avoid the collinearity restriction when incorporating subspace recycling into solving systems of the form (\ref{eq:shiftedsystem}).  The method uses a projected Krylov subspace and works by exploiting the equivalence of the shifted system to a Sylvester equation, and using  the shift invariance of the block Krylov subspace generated by the Sylvester operator.
This avoids the need for the residuals to span the same subspace after each cycle and allows for a viable restarted GMRES algorithm with subspace recycling for shifted systems with unrelated right hand sides. It also allows one to exploit the usual 
computational efficiencies associated with block Krylov subspace methods ( see e.g. \cite{MR3707736}).

The purpose of this paper is to study methods such as rsbGMRES more generally using a well known residual 
projection framework describing augmented Krylov subspace methods. This framework was developed in works such as  \cite{gaul2014recycling, MR3054589, MR2995778, MR3284571} and later generalized in a recent survey on subspace recycling iterative methods \cite{MR4175151}.  The original formulation of the framework describes augmented Krylov subspace methods for linear systems of the form (cf. \ref{eq:linearsystem}) in terms of applying a standard Krylov subspace method to some projected problem, followed by taking an additional correction from the augmentation subspace. Despite its original formulation, recent work \cite{burke2022augmented} has shown that the framework does not need to be viewed in terms of solving a projected problem. An equivalent unprojected formulation of the problem can also be solved and viewing the framework in this way allows one to describe the class of so called \textit{unprojected} augmented Krylov subspace methods.

In this paper we extend the framework to the case of shifted linear systems of the form (\ref{eq:shiftedsystem}) using block Krylov subspaces. By doing so, we show how a practical methods based on projected Krylov subspaces exploiting the ideas presented in \cite{MR3530183} cannot fit within the framework. This is because a shift dependent projector arises from the framework and due to increased storage and computation requirements, a practical method cannot use a separate Krylov subspace for every shift. We discuss how methods such as rsbGMRES which use projected Krylov subspaces can attempt to overcome these difficulties, but in particular we focus on developing methods that fit within the general framework. This serves multiple purposes (see e.g. \cite{MR4175151}).
\begin{itemize}
    \item It allows one to derive new types of methods for certain types of problems more easily by simply selecting appropriate subspaces.
    
    \item It allows one to determine the appropriate projector for the method and this in turn allows one to understand the best conditions to impose on certain basis vectors (e.g. orthonormality) to simplify the structure of the projections. 
\end{itemize}
The main purpose of this paper is to show how the unprojected formulation of the framework \cite{burke2022augmented} allows one to avoid the difficulties with developing recycled shifted block  methods based on projected Krylov subspaces. As such, we show how practical methods based on \textit{unprojected} Krylov subspaces can easily fit within the framework. This gives us a framework for understanding recycled block Krylov subspace methods for computing the solution to a family of shifted linear systems with unrelated right hand sides. We then take advantage of the above properties and easily derive new  methods such as an unprojected rsbFOM and rsbGMRES algorithm. In addition our use of unprojected augmented Krylov subspace methods reduces the amount of times the projector needs to be applied in each cycle of the algorithm. This serves as a computational advantage in situations where the projector may be too expensive to apply and there is a strict constraint on the amount of computations one can perform per iteration \cite{MR1780273,ramlau2020augmented, MR4449205}. This makes our unprojected methods of significant interest in the context of High-Performance Computing applications.

\subsection{Layout of paper} The layout of this paper is as follows. In \Cref{sec:intro_to_Krylov} and \Cref{sec:intro_to_block_krylov} we give a brief introduction to Krylov subspace methods and their extension to block Krylov subspace methods for systems with multiple right hand sides. In \Cref{sec:sbFOMandsbGMRES} we introduce the sbFOM and sbGMRES algorithms which provide a means to avoid the collinearity restriction for shifted systems using block Krylov subspaces. In \Cref{sec:framework_extension} we extend the augmented Krylov subspace framework in \cite{MR3054589, MR4175151} to the case of solving shifted systems of the form \cref{eq:shiftedsystem} using block Krylov subspace methods. This allows us to understand the difficulties associated to developing practical implementations of these methods based on projected Krylov subspaces and as such show how they cannot fit within the framework. We then discuss how rsbGMRES attempts to overcome these difficulties. In \Cref{sec:unprojrsbFOMandGMRES} we show how the unprojected formulation of the framework from \cite{burke2022augmented} allows one to overcome these practical limitations while also allowing practical implementations of recycled shifted block methods to fit within the framework. We derive an unprojected rsbFOM and unprojected rsbGMRES algorithm. In \Cref{sec:harmritz} we derive an shift dependent harmonic Ritz procedure for our methods. Numerical experiments to demonstrate the effectiveness of our methods are presented in \Cref{sec:numerical_experiments}. We finish by giving brief conclusions in \Cref{sec:conclusions}.

\section{An introduction to Krylov subspace methods}\label{sec:intro_to_Krylov}
Krylov subspace methods are a class of iterative methods for solving linear systems of the form 
\begin{equation}\label{eq:linearsystem}
     \textbf{A}\textbf{x} = \textbf{b}
\end{equation}
when the matrix $\textbf{A} \in \mathbb{C}^{n \times n}$ is large and sparse. They work by projecting the problem onto the $j$ dimensional Krylov subspace
\[
     \mathcal{K}_{j}(\textbf{A},\textbf{r}_{0}) = \{ \textbf{r}_{0}, \textbf{A}\textbf{r}_{0}, \textbf{A}^{2}\textbf{r}_{0},...,\textbf{A}^{j-1}\textbf{r}_{0} \},
\]
where $\textbf{r}_{0} = \textbf{b} - \textbf{A}\textbf{x}_{0}$ is the initial residual corresponding to an initial approximation $\textbf{x}_{0}$. A smaller representation of the problem can then be solved efficiently on this subspace via a standard direct method (see e.g. \cite{MR1990645}).

The main kernel of any Krylov subspace method involves building a an orthonormal basis for the Krylov subspace $\mathcal{K}_{j}(\textbf{A},\textbf{r}_{0})$. This is commonly done using a stable version of a Gram-Schmidt process known as \textit{modified} Gram-Schmidt. At iteration $j$ we compute a new basis vector using an application of $\textbf{A}$ to most recently constructed basis vector. We then orthogonalize this vector against all previously constructed basis vectors. The resulting algorithm is known as the Arnoldi algorithm. It is outlined in \Cref{alg:arnoldi}.
\begin{algorithm}[H]
\caption{Arnoldi with Modified Gram-Schmidt \label{alg:arnoldi}}
\begin{algorithmic}[1]
\STATE{\textbf{Input:} $\textbf{A}$, $\textbf{r}_{0}$, cycle length $m$}
\STATE{Compute $\textbf{v}_{1} \leftarrow \frac{\textbf{r}_{0}}{\| \textbf{r}_{0}\|}$ }
\FOR{$j = 1,2,...m$}
\STATE{$\textbf{w} \leftarrow \textbf{A} \textbf{v}_{j}$}
\FOR{$i=1,2,..,j$}
\STATE{$h_{i,j} \leftarrow \textbf{v}_{i}^{*} \textbf{w} $}
\STATE{$\textbf{w} \leftarrow \textbf{w} - h_{i,j} \textbf{v}_{i}$}
\ENDFOR
\STATE{ $h_{j+1,j} \leftarrow \| \textbf{w} \|_{2}$ }
\STATE{ $\textbf{v}_{j+1} \leftarrow \textbf{w} / h_{j+1,j}$}
\ENDFOR
\end{algorithmic}
\end{algorithm}
After iteration $j$ of the Arnoldi algorithm an important relation known as the \textit{Arnoldi relation} is satisfied
\[
    \textbf{A}\textbf{V}_{j} = \textbf{V}_{j+1} \overline{\textbf{H}}_{j} = \textbf{V}_{j}\textbf{H}_{j} + h_{j+1,j} \textbf{v}_{j+1} \textbf{e}_{j}^{T}.
\]
The matrix $\textbf{V}_{j+1} \in \mathbb{C}^{n \times (j+1)}$ stores the basis vectors and the matrix $\overline{\textbf{H}}_{j} \in \mathbb{C}^{(j+1) \times j}$ is an upper Hessenberg matrix storing the orthogonalization coefficients. The matrix $\textbf{H}_{j} \in \mathbb{C}^{j \times j}$ is simply the matrix $\textbf{H}_{j}$ with last row deleted.

At iteration $j$ a correction $\textbf{t}_{j} \in \mathcal{K}_{j}(\textbf{A},\textbf{r}_{0})$ is selected according to some constraint on the residual. The choice of residual constraint space determines the method. The Full Orthogonalization Method (FOM) selects $\textbf{t}_{j}$ such that
\[
     \textbf{r}_{j} = \textbf{b} - \textbf{A}(\textbf{x}_{0} + \textbf{t}_{j}) \perp \mathcal{K}_{j}(\textbf{A},\textbf{r}_{0}).
\]
This leads to a projection of the original problem down to the smaller problem of solving the following $j \times j$ linear system for the vector $\textbf{y}_{j} \in \mathbb{C}^{j}$
\[
     \textbf{H}_{j} \textbf{y}_{j} = \| \textbf{r}_{0} \| \textbf{e}_{1}.
\]

A related method known as the Generalized Minimum Residual Method (GMRES) chooses $\textbf{t}_{j}$ according to the constraint 
\[
     \textbf{r}_{j} = \textbf{b} - \textbf{A}(\textbf{x}_{0} + \textbf{t}_{j}) \perp \textbf{A}\mathcal{K}_{j}(\textbf{A},\textbf{r}_{0}).
\]
The leads to a small $(j+1) \times j$ least squares problem for $\textbf{y}_{j}$ given by
\[
  \textbf{y}_{j} = \operatorname*{argmin}_{\textbf{z} \in \mathbb{C}^{j}}  \left\| \| \textbf{r}_{0} \| \textbf{e}_{1} - \overline{\textbf{H}}_{j} \textbf{z} \right\|.
\]

Once we computed $\textbf{y}_{j}$ we can compute the new approximation at step $j$ as 
\[
\textbf{x}_{j} = \textbf{x}_{0} + \textbf{V}_{j} \textbf{y}_{j}.
\]
The residual can then be updated cheaply via the Arnoldi relation by
\[
    \textbf{r}_{j} = \textbf{r}_{0} - \textbf{V}_{j+1} \overline{\textbf{H}}_{j}\textbf{y}_{j}.
\]
Due to the fact that the basis grows in size at each iteration, memory constraints make restarted Krylov subspace methods unavoidable. In a restarted Krylov subspace method we run a $j$ length Arnoldi cycle and compute a new approximation $\textbf{x}_{j}$ and associated residual $\textbf{r}_{j}$.  We then discard the basis $\textbf{V}_{j+1}$ and the matrix $\overline{\textbf{H}}_{j}$ and run a new cycle with initial approximation $\textbf{x}_{j}$ using the Krylov subspace $\mathcal{K}(\textbf{A},\textbf{r}_{j})$. We repeat this process until a desired level of convergence is met or until a user specified maximum amount of cycles has completed. The restarted FOM and GMRES algorithms are outlined in \Cref{alg:restartedFOMandGMRES}.

\begin{algorithm}[H]
\caption{Restarted FOM and GMRES
\label{alg:restartedFOMandGMRES}} 
\begin{algorithmic}[1]
\STATE{\textbf{Input:} $\textbf{A} \in \mathbb{C}^{n \times n}$,  $\textbf{b} \in \mathbb{C}^{n}$, $\textbf{x}_{0} \in \mathbb{C}^{n}$, tolerance $\epsilon$}
\STATE{$\textbf{r}_{0} = \textbf{b} - \textbf{A} \textbf{x}_{0}$}
\WHILE{$\|\textbf{r}_{0}\| > \epsilon$}
\STATE{Build a basis for the Krylov subspace $\mathcal{K}_{j}(\textbf{A},\textbf{r}_{0})$ via the Arnoldi algorithm generating $\textbf{V}_{j+1} \in \mathbb{C}^{n \times (j+1)}$ and $\overline{\textbf{H}}_{j} \in \mathbb{C}^{(j+1) \times j}$.}
  \IF{FOM}
      \STATE{Solve $ \textbf{H}_{j} \textbf{y} = \| \textbf{r}_{0} \|\textbf{e}_{1} $ for $\textbf{y}$.}
  \ENDIF
  \IF{GMRES}
     \STATE{$\textbf{y}_{j} \leftarrow \operatorname*{argmin}_{\textbf{z} \in \mathbb{C}^{j}}  \left\| \| \textbf{r}_{0} \| \textbf{e}_{1} - \overline{\textbf{H}}_{j} \textbf{z} \right\| $ }
  \ENDIF
  \STATE{$\textbf{x}_{j}  \leftarrow \textbf{x}_{0} + \textbf{V}_{j}\textbf{y}$}
  \STATE{$\textbf{r}_{j} \leftarrow \textbf{r}_{0} - \textbf{V}_{j+1} \overline{\textbf{H}}_{j} \textbf{y}$}
\ENDWHILE
\end{algorithmic}
\end{algorithm}

\section{Block Krylov subspace methods}\label{sec:intro_to_block_krylov}
Consider the case where one has a linear system with several right hand sides available simultaneously $\textbf{b}^{(1)}, \textbf{b}^{(2)},.., \textbf{b}^{(s)} \in \mathbb{C}^{n}$. In this setting one can put all the right hand sides as columns into a single matrix
\[
     \textbf{B} \leftarrow \begin{bmatrix}
           \textbf{b}^{(1)} & \textbf{b}^{(2)} & ... & \textbf{b}^{(s)}
     \end{bmatrix}  \in \mathbb{C}^{n \times s},
\]
and compute the solution approximation $\textbf{X}_{j} \in \mathbb{C}^{n \times s}$ to the corresponding block linear system $\textbf{A}\textbf{X} = \textbf{B}$ using a generalization of Krylov subspace methods to the so called \textit{block} Krylov subspace methods. Beginning with an initial solution approximation $\textbf{X}_{0}$ and corresponding block residual $\textbf{R}_{0} = \textbf{B} - \textbf{A}\textbf{X}_{0}$ one can build a basis for the block Krylov subspace
\begin{align*}
    \mathcal{K}_{j}(\textbf{A},\textbf{R}_{0}) = \mbox{colspan}\{ \textbf{R}_{0}, \textbf{A}\textbf{R}_{0}, \textbf{A}^{2}\textbf{R}_{0}, .. , \textbf{A}^{j-1} \textbf{R}_{0}  \} \\ = \mathcal{K}_{j}(\textbf{A},\textbf{r}_{0}^{(1)}) + \mathcal{K}_{j}(\textbf{A},\textbf{r}_{0}^{(2)}) + ... +
    \mathcal{K}_{j}(\textbf{A},\textbf{r}_{0}^{(s)})
\end{align*}
where $\textbf{r}_{0}^{(i)} \in \mathbb{C}^{n}$ is the $i^{th}$ column of the block vector $\textbf{R}_{0}$. We use the $j$ subscript to denote the fact the subspace is of dimension $j s$ where $s$ is is the number of right hand sides.
An orthonormal basis for the block Krylov subspace $\mathcal{K}_{j}(\textbf{A},\textbf{R}_{0})$ can be built using a block extension of the Arnoldi algorithm. This time, the matrix $\textbf{A}$ is applied to a normalized block vector at each iteration. The normalization step of the block Arnoldi algorithm now involves computing a reduced QR factorization of an $n \times s$ block vector. The normalized starting block vector $V_{1} \in \mathbb{C}^{n \times s}$ can be computed using the reduced QR factorization $\textbf{R}_{0} = V_{1} \textbf{S}_{0}$ where  $\textbf{S}_{0} \in \mathbb{C}^{s \times s}$ is an upper triangular matrix.
We denote each block element generated by the block Arnoldi process as a capital letter. At iteration $j$ a new block vector $V_{j}$ is computed and collected as block columns of the matrix
$\textbf{W}_{j+1} = \begin{bmatrix}
      V_{1} & V_{2} & ... & V_{j+1}
\end{bmatrix} \in \mathbb{C}^{n \times (j+1)s}$.
The corresponding  block Hessenberg matrix $\overline{\textbf{H}}_{j} \in \mathbb{C}^{(j+1)s \times js}$ is also generated.  The elements of this matrix are now $s \times s $ blocks $H_{i,j} \in \mathbb{C}^{s \times s}$.

\begin{algorithm}[H]
\caption{Block Arnoldi with Modified Block Gram-Schmidt \label{alg:block_arnoldi}}
\begin{algorithmic}[1]
\STATE{\textbf{Input:} $\textbf{A}$, $\textbf{R}$, cycle length $m$}
\STATE{Compute the reduced QR decomposition $\textbf{R} = V_{1} \textbf{S}_{0}$}
\FOR{$j = 1,2,...,m$}
\STATE{$W \leftarrow \textbf{A} V_{j}$}
\FOR{$i=1,2,..,j$}
\STATE{$H_{i,j} \leftarrow V_{i}^{*} W $}
\STATE{$W \leftarrow W - V_{i} H_{i,j}$}
\ENDFOR
\STATE{Compute the reduced QR decomposition $W \leftarrow V_{j+1}H_{j+1,j}$}
\ENDFOR
\end{algorithmic}
\end{algorithm}

After iteration $j$ of the block Arnoldi process we have the block Arnoldi relation 
\begin{equation}\label{eq:block_Arnoldi}
    \textbf{A} \textbf{W}_{j} = \textbf{W}_{j+1} \overline{\textbf{H}}_{j} = \textbf{W}_{j} \textbf{H}_{j} + V_{j+1} H_{j+1,j} \textbf{E}_{j}^{T},
\end{equation}
where $\textbf{E}_{j} \in \mathbb{C}^{j s \times s}$ is a block vector of zeros with the $s \times s$ identity matrix contained in the first $s$ rows. The matrix $\textbf{H}_{j} \in \mathbb{C}^{j s \times j s}$ is now the matrix $\overline{\textbf{H}}_{j}$ with the last $s$ rows deleted. The block Arnoldi algorithm allows for an extension of FOM and GMRES to \textit{block} FOM and \textit{block} GMRES respectively \cite{MR1990645}. These methods require one to solve a problem for the block vector $\textbf{Y}_{j} \in \mathbb{C}^{j s \times s}$. For block FOM this problem is
\[
       \textbf{H}_{j} \textbf{Y}_{j} = \textbf{E}_{j} \textbf{S}_{0},
\]
and for block GMRES it is
\[
  \textbf{Y}_{j} = \operatorname*{argmin}_{\textbf{Z} \in \mathbb{C}^{sj \times s}}  \left \| \textbf{E}_{j+1} \textbf{S}_{0} - \overline{\textbf{H}}_{j} \textbf{Z} \right\|.
\]
The block solution approximation can then be computed via 
\[
     \textbf{X}_{j} = \textbf{X}_{0} + \textbf{W}_{j} \textbf{Y}_{j}.
\]

Block Krylov subspaces have many advantages (see e.g. \cite{MR2206505,MR3707736}). Instead of solving for every right hand side separately using $s$ single standard Krylov subspace methods, we can solve for them all simultaneously using block Krylov subspace methods. In addition block Krylov subspace methods have added computational efficiency since the matrix $\textbf{A}$ needs to only be accessed once per step of the Arnoldi process as opposed to $s$ times. In addition, block Krylov subspace methods construct a larger subspace from which the approximation is drawn for each system, leading to better approximations. The main advantage exploited in this work is the fact the block Krylov subspace is shift invariant
\begin{equation}\label{eq:block_shift_invariance}
    \mathcal{K}_{j}(\textbf{A},\textbf{R}_{0}) = \mathcal{K}_{j}(\textbf{A}+\sigma \textbf{I},\textbf{R}_{0})
\end{equation}
with $\sigma \in \mathbb{C}\backslash \{ 0 \}$, and this relation holds without any collinearity condition. With this invariance comes the shifted block Arnoldi relation
\[
    (\textbf{A} + \sigma \textbf{I}) \textbf{W}_{j} = \textbf{W}_{j+1}(\overline{\textbf{H}}_{j} + \sigma \overline{\textbf{I}}).
\]
In this case $\overline{\textbf{I}} \in \mathbb{R}^{(j+1)s \times j s}$ is simply a $j s \times j s$ identity matrix padded with an additional $s$ rows of zeros underneath. Before we consider recycling, we discuss the use of block Krylov subspace methods for avoiding the collinearity restriction when solving shifted systems by exploiting the shift invariance \cref{eq:block_shift_invariance}. The resulting algorithms are known as shifted block FOM (sbFOM) and shifted block GMRES (sbGMRES) \cite{MR1420277, MR3532794}.

\section{sbFOM and sbGMRES}\label{sec:sbFOMandsbGMRES}
We will drop the system index notation in  (\ref{eq:shiftedsystem}) and use a superscript $^{(i)}$ to denote vectors corresponding to shift $\sigma_{i}$. We first discuss shifted FOM and GMRES for computing the solution approximation for shift $\sigma_{i}, \textbf{x}_{j}^{(i)}$ for systems of the form
\begin{equation}\label{eq:no_superscript_shifted_system}
     (\textbf{A} + \sigma_{i} \textbf{I}) \textbf{x}^{(i)} = \textbf{b}^{(i)} \hspace{1cm} i = 1,2, .. , s
\end{equation}
with $\sigma_{i},\sigma_{j} \in \mathbb{C}$. 

The use of Krylov subspace methods is an attractive means of solving such systems provided that the initial residuals for each shift $\sigma_{i}$ span the same subspace and satisfy $\textbf{r}_{0}^{(i)} = \beta \textbf{r}_{0}^{(j)}, j \neq i$ where $\textbf{r}_{0}^{(i)}$ is the initial residual corresponding to shift $i$. If this holds true for all $s$ shifts then the \textit{shift invariance} property is satisfied
\[
\mathcal{K}_{j}(\textbf{A} + \sigma_{i} \textbf{I}, \textbf{r}^{(i)}_{0}) = \mathcal{K}_{j}(\textbf{A} + \sigma_{j} \textbf{I}, \textbf{r}_{0}^{(j)}).
\]
This means one can compute the solution for all shifts over the single Krylov subspace $\mathcal{K}_{j}(\textbf{A},\textbf{r}_{0}^{(1)})$ where $\textbf{r}_{0}^{(1)}$ is the initial residual corresponding to the base shift.
We can use the same basis $\textbf{V}_{j+1}$ for all shifts which now satisfy the shifted Arnoldi relation
\[
    (\textbf{A} + \sigma_{i} \textbf{I}) \textbf{V}_{j} = \textbf{V}_{j+1}(\overline{\textbf{H}}_{j} + \sigma_{i} \overline{\textbf{I}} ) = \textbf{V}_{j}(\textbf{H}_{j} + \sigma_{i} \textbf{I}) + h_{j+1,j} \textbf{v}_{j+1} \textbf{e}_{j}^{T},
\]
where the matrix $\overline{\textbf{I}} \in \mathbb{C}^{(j+1) \times j}$ is the $j \times j$ identity matrix padded with an additional row of zeros in the bottom row. Similar to the non-shifted case, this Arnoldi relation allows one to derive a shifted FOM approximation for shift $i$ 
\[
    \textbf{y}_{j}^{(i)} = (\textbf{H}_{j} + \sigma_{i} \textbf{I})^{-1} (\| \textbf{r}_{0} \| \textbf{e}_{1}),
\]
where $\| \textbf{r}_{0} \| $ is the norm of the initial residual corresponding to the base shift.

A shifted GMRES approximation for shift $i$ can also be derived to give
\[
  \textbf{y}_{j}^{(i)} = \operatorname*{argmin}_{\textbf{z} \in \mathbb{C}^{j}}  \left\| \| \textbf{r}_{0} \| \textbf{e}_{1} - (\overline{\textbf{H}}_{j} + \sigma_{i} \textbf{I}) \textbf{z} \right\|.
\]
The solution approximation at step $j$ for shift $i$ can then be updated in a similar way as before by
\[
    \textbf{x}_{j}^{(i)} = \textbf{x}_{0}^{(i)} + \textbf{V}_{j} \textbf{y}_{j}^{(i)},
\]
and the residual update for shift $i$ is
\[
    \textbf{r}_{j}^{(i)} = \textbf{r}_{0}^{(i)} - \textbf{V}_{j+1} (\overline{\textbf{H}}_{j} + \sigma_{i} \overline{\textbf{I}})\textbf{y}_{j}^{(i)}.
\]
After a cycle of the shifted FOM algorithm \cite{MR2011689} the newly update residuals are naturally collinear and thus the shifted FOM algorithm can easily be restarted. This is not true for GMRES since the residuals are generally no longer collinear after restart. Methods do exist which enforce collinearity allowing for a viable restarted method for shifted systems \cite{MR1616874}. Another way to avoid the collinearity restriction completely is to use block Krylov subspace methods. We discuss this in the context of shifted block FOM (sbFOM) and shifted block GMRES (sbGMRES) \cite{MR1420277, MR3532794}. The following discussion is primarily based on \cite{MR3530183}.

We note from works such as \cite{MR1909198, MR1420277, MR3532794, MR3530183} that the shifted system \cref{eq:no_superscript_shifted_system} is equivalent to the following Sylvester equation
\[
     \textbf{A} \textbf{X} + \textbf{X} \textbf{D} = \textbf{B},
\]
where $\textbf{D} = \mbox{diag}\{ \sigma_{1}, \sigma_{2},..,\sigma_{s}\} \in \mathbb{C}^{s \times s}$, and $\textbf{X},\textbf{B} \in \mathbb{C}^{n \times s}$ are the block vectors containing the solutions and right hand sides respectively. It has been shown in \cite{MR1909198} that for a vector block $\textbf{F} \in \mathbb{C}^{n \times s}$ if we define the Sylvester operator as 
\[
    \mathcal{T} : \textbf{F} \rightarrow \textbf{A} \textbf{F} + \textbf{F} \textbf{D},
\]
and define powers of the operator as
\[
   \mathcal{T}^{i}\textbf{F} = \mathcal{T} (\mathcal{T}^{i-1} \textbf{F}),
\]
with $\mathcal{T}^{0} \textbf{F} = \textbf{F}$, then the block Krylov subspace generated by $\mathcal{T}$ and $\textbf{F}$ is equivalent to that generated by the matrix $\textbf{A}$ and $\textbf{F}$. In particular, we have 
\begin{equation}\label{eq:sylvester_invariance}
\mathcal{K}_{j}(\mathcal{T},\textbf{F}) = \mathcal{K}_{j}(\textbf{A}, \textbf{F}).
\end{equation}
With initial approximation $\textbf{X}_{0}$ and corresponding residual $\textbf{R}_{0} = \textbf{B} - (\textbf{A}\textbf{X}_{0} + \textbf{X}_{0}\textbf{D})$, one can generate the block Krylov subspace $\mathcal{K}_{j}(\textbf{A},\textbf{R}_{0})$ and project the problem onto this subspace and solve using a FOM or GMRES type method for shifted systems over a block Krylov subspace. The shift invariance of the block Krylov subspace \cref{eq:block_shift_invariance} and the equivalence of the block Krylov subspace to that generated by the Sylvester operator \cref{eq:sylvester_invariance} allows such methods to completely avoid the need for collinear residuals. 

A shifted block FOM (sbFOM) method for the shifted system \cref{eq:no_superscript_shifted_system} can be derived  by imposing the following FOM constraint on the residual of the $i^{th}$ system
\[
    \textbf{r}_{j}^{(i)} = \textbf{b}^{(i)} - (\textbf{A} + \sigma_{i} \textbf{I})(\textbf{x}_{0}^{(i)} + \textbf{t}_{j}^{(i)}) \perp \mathcal{K}_{j}(\textbf{A},\textbf{R}_{0}).
\]
This leads to the following $j s \times j s$ linear system for $\textbf{y}_{j}^{(i)}$
\[
    (\textbf{H}_{j} + \sigma_{i} \textbf{I}) \textbf{y}_{j}^{(i)} = \textbf{E}_{j} \textbf{S}_{0} \textbf{e}_{i}.
\]
Similarly a shifted block GMRES (sbGMRES) algorithm can be derived using 
\[
    \textbf{r}_{j}^{(i)} = \textbf{b}^{(i)} - (\textbf{A} + \sigma_{i} \textbf{I})(\textbf{x}_{0}^{(i)} + \textbf{t}_{j}^{(i)}) \perp \textbf{A} \mathcal{K}_{j}(\textbf{A},\textbf{R}_{0}).
\]
This leads to the following minimization problem for shift $i$.
\[
\textbf{y}^{(i)}_{j} = \operatorname*{argmin}_{\textbf{z} \in \mathbb{C}^{js}}  \left\| \textbf{E}_{j+1} \textbf{S}_{0} \textbf{e}_{i} - (\overline{\textbf{H}}_{j} + \sigma_{i}\overline{\textbf{I}}) \textbf{z} \right\|.
\]

Since sbFOM and sbGMRES are based on block Krylov subspace methods it is important to ensure the initial residuals at the beginning of each cycle are linearly independent. This avoids numerical breakdowns in the block Arnoldi algorithm. There is different ways one can render initial residuals non-collinear \cite{MR3530183}. In this paper we simply propose to take a random initial approximation for $\textbf{X}_{0}$. Throughout this paper we will assume there is no breakdowns in the block Arnoldi process and that $\mbox{dim}(\mathcal{K}_{j}(\textbf{A},\textbf{R}_{0})) = j s$. The sbFOM and sbGMRES algorithms are outlined in \Cref{alg:sbFOMandsbGMRES}.

\begin{algorithm}[H]
\caption{sbFOM and sbGMRES
\label{alg:sbFOMandsbGMRES}} 
\label{version1:alg}
\begin{algorithmic}[1]
\STATE{\textbf{Input:} $\textbf{A} \in \mathbb{C}^{n \times n}$,  $\textbf{b}^{(1)}, \textbf{b}^{(2)}, ..., \textbf{b}^{(s)} \in \mathbb{C}^{n}$, $\textbf{x}_{0}^{(1)}, \textbf{x}_{0}^{(2)},...,\textbf{x}_{0}^{(s)} \in \mathbb{C}^{n}$, shifts $\sigma_{1},..,\sigma_{s} \in \mathbb{C}$, tolerance $\epsilon$, Arnoldi cycle length $j$.}
\FOR{$i=1,2,..,s$}
  \STATE{$\textbf{r}^{(i)}_{0} = \textbf{b}^{(i)} - (\textbf{A}+\sigma_{i}\textbf{I}) \textbf{x}^{(i)}_{0}$}
\ENDFOR
\IF{$\textbf{r}_{i} = \beta \textbf{r}_{k}$ for some $i \neq k$ and $\beta \in \mathbb{C}$}
\STATE{Render residuals non-collinear.}
\ENDIF
\STATE{Set $\textbf{R} \leftarrow \begin{bmatrix} \textbf{r}_{0}^{(1)} & \textbf{r}_{0}^{(2)} & ... & \textbf{r}_{0}^{(s)} \end{bmatrix}$}
\WHILE{$\|\textbf{R}\| > \epsilon$}
\STATE{Build a basis for the block Krylov subspace $\mathcal{K}_{j}(\textbf{A},\textbf{R})$ via the block Arnoldi algorithm generating $\textbf{W}_{j+1} \in \mathbb{C}^{n \times (j+1)s}$ and $\overline{\textbf{H}}_{j} \in \mathbb{C}^{(j+1)s \times js}$.}
\FOR{$i = 1,.., s$}
  \IF{FOM}
      \STATE{Solve $ (\textbf{H}_{j} + \sigma_{i} \textbf{I}) \textbf{y} = \textbf{E}_{j} \textbf{S}_{0} \textbf{e}_{i}$ for $\textbf{y}$.}
  \ENDIF
  \IF{GMRES}
     \STATE{$\textbf{y} = \operatorname*{argmin}_{\textbf{z} \in \mathbb{C}^{js}}  \left\| \textbf{E}_{j+1} \textbf{S}_{0} \textbf{e}_{i} - (\overline{\textbf{H}}_{j} + \sigma_{i}\overline{\textbf{I}}) \textbf{z} \right\|$}
  \ENDIF
  \STATE{$\textbf{x}_{j}^{(i)}  \leftarrow \textbf{x}_{0}^{(i)} + \textbf{W}_{j}\textbf{y}$}
  \STATE{$\textbf{r}_{j}^{(i)} \leftarrow \textbf{r}_{0}^{(i)} - \textbf{W}_{j+1} (\overline{\textbf{H}}_{j} + \sigma_{i} \overline{\textbf{I}}) \textbf{y}$}
\ENDFOR
\STATE{$\textbf{R} \leftarrow \begin{bmatrix} \textbf{r}_{j}^{(1)} & \textbf{r}_{j}^{(2)} & ... & \textbf{r}_{j}^{(s)}\end{bmatrix}$}
\ENDWHILE
\end{algorithmic}
\end{algorithm}

We now show how one can use the same idea as sbFOM and sbGMRES to develop a \textit{recycled} Krylov subspace method for shifted systems with unrelated right hand sides. This was done in \cite{MR3530183} where the rsbGMRES algorithm was derived without using the framework \cite{MR3054589,MR4175151}.  Before discussing the details of rsbGMRES, we will attempt to extend the framework to the class of recycled shifted block methods. We will follow a similar approach as in \cite{MR4175151} and do so in terms of arbitrary search and constraint spaces. We will use the framework  to understand how to best derive a recycled shifted block method.

\section{A derivation of an augmented Krylov subspace approximation to shifted linear systems with block Krylov subspaces}\label{sec:framework_extension}
We present an extension of the framework to the case of shifted linear systems using block Krylov subspaces in an attempt to derive a recycled block Krylov subspace method for shifted systems with linearly independent right hand sides. In practice we must use the same $j s$ dimensional block Krylov subspace $\mathcal{W}_{j}$ for every shift since building a separate basis for every shift is not practical or efficient in terms of computation or storage requirements. We assume an initial approximation for all shifts $\textbf{x}_{0}^{(i)}, i = 1,2,..,s$ is available and that a solution correction for shift $i$ at step $j$ will be taken from the  block Krylov subspace $ \textbf{t}_{j}^{(i)} \in \mathcal{W}_{j} $ and a fixed $k$ dimensional augmentation subspace $\textbf{s}_{j}^{(i)} \in \mathcal{U}$. This yields an approximation for shift $i$ defined by
\[
    \textbf{x}_{j}^{(i)} = 
 \textbf{x}_{0}^{(i)} + \textbf{s}_j^{(i)} + \textbf{t}_{j}^{(i)}
= \textbf{x}_{0}^{(i)} +    \textbf{U} \textbf{z}_{j}^{(i)} + \textbf{W}_{j} \textbf{y}_{j}^{(i)}
\]
with $\textbf{z}_{j}^{(i)} \in \mathbb{C}^{k}$ and $\textbf{y}_{j}^{(i)} \in \mathbb{C}^{j s}$ to be determined. This is done by imposing the residual for shift $i$ be orthogonal to the sum of some other $j s$ dimensional block Krylov subspace $\widetilde{\mathcal{W}}_{j}$ and some other augmentation subspace $\widetilde{\mathcal{U}}$
\[
    \textbf{r}_{j}^{(i)} = \textbf{b} - (\textbf{A} + \sigma_{i} \textbf{I})( \textbf{x}_{0}^{(i)} + \textbf{U} \textbf{z}_{j}^{(i)} + \textbf{W}_{j} \textbf{y}_{j}^{(i)}   ) \perp  \widetilde{\mathcal{U}} + \widetilde{\mathcal{W}}_{j}.
\]
 The residual constraint yields the following linear system for $\textbf{z}_{j}^{(i)}$ and $\textbf{y}_{j}^{(i)}$
\[
\begin{bmatrix} \widetilde{\textbf{U}}^{*}(\textbf{A} + \sigma_{i} \textbf{I}) \textbf{U} &    \widetilde{\textbf{U}}^{*}(\textbf{A} + \sigma \textbf{I}) \textbf{W}_{j} \\
   \widetilde{\textbf{W}}^{*}_{j}(\textbf{A} + \sigma_{i} \textbf{I}) \textbf{U} &  \widetilde{\textbf{W}}^{*}_{j}(\textbf{A} + \sigma_{i} \textbf{I}) \textbf{W}_{j}
  \end{bmatrix} \begin{bmatrix}
        \textbf{z}_{j}^{(i)} \\
        \textbf{y}_{j}^{(i)} 
    \end{bmatrix} = \begin{bmatrix}
        \widetilde{\textbf{U}}^{*} \textbf{r}_{0}^{(i)} \\
       \widetilde{\textbf{W}}_{j}^{*} \textbf{r}_{0}^{(i)}
  \end{bmatrix}.
\]
A block LU factorization then yields a new linear system for $\textbf{z}_{j}^{(i)}$ and $\textbf{y}_{j}^{(i)}$ with coefficient matrix
\begin{align*}
 \begin{bmatrix}
        \widetilde{\textbf{U}}^{*}( \textbf{A} + \sigma_{i} \textbf{I}) \textbf{U} &  \widetilde{\textbf{U}}^{*}( \textbf{A} + \sigma_{i} \textbf{I}) \textbf{W}_{j} \\ 0 &  \widetilde{\textbf{W}}^{*}_{j}( \textbf{A} + \sigma_{i} \textbf{I}) \textbf{W}_{j} - \widetilde{\textbf{W}}^{*}_{j} ( \textbf{A} + \sigma_{i} \textbf{I}) \textbf{U} (\widetilde{\textbf{U}}^{*} ( \textbf{A} + \sigma_{i} \textbf{I}) \textbf{U})^{-1} \widetilde{\textbf{U}}^{*}( \textbf{A} + \sigma_{i} \textbf{I}) \textbf{W}_{j}
    \end{bmatrix} 
\end{align*}
and right hand side
\begin{align*}
  \begin{bmatrix}
        \widetilde{\textbf{U}}^{*} \textbf{r}^{(i)} \\
        \widetilde{\textbf{W}}^{*}_{j} \textbf{r}_{0}^{(i)} - \widetilde{\textbf{W}}^{*}_{j} ( \textbf{A} + \sigma_{i} \textbf{I}) \textbf{U}(\widetilde{\textbf{U}}^{*} ( \textbf{A} + \sigma_{i} \textbf{I}) \textbf{U})^{-1}\widetilde{\textbf{U}}^{*} \textbf{r}_{0}^{(i)}
     \end{bmatrix}.
\end{align*}
The first equation yields a means of computing the correction $\textbf{z}_{j}^{(i)}$ in terms of $\textbf{y}_{j}^{(i)}$ 
\[
    \textbf{z}_{j}^{(i)} = (\widetilde{\textbf{U}}^{*} (\textbf{A} + \sigma_{i} \textbf{I}) \textbf{U})^{-1}\widetilde{\textbf{U}}^{*} \textbf{r}_{0}^{(i)} -  (\widetilde{\textbf{U}}^{*} (\textbf{A} + \sigma_{i} \textbf{I}) \textbf{U})^{-1}\widetilde{\textbf{U}}^{*} (\textbf{A} + \sigma_{i} \textbf{I}) \textbf{W}_{j} \textbf{y}_{j}^{(i)},
\]
and the second equation then gives a single equation for $\textbf{y}_{j}^{(i)}$ itself
\begin{equation}
    \widetilde{\textbf{W}}_{j}^{*} (\textbf{I} - \Phi^{(i)})(\textbf{A} + \sigma_{i} \textbf{I}) \textbf{W}_{j} \textbf{y}_{j}^{(i)} =  \widetilde{\textbf{W}}_{j}^{*} (\textbf{I} - \Phi^{(i)}) \textbf{r}_{0}^{(i)},
\end{equation}
where  $\Phi^{(i)} := (\textbf{A} + \sigma_{i} \textbf{I}) \textbf{U} (\widetilde{\textbf{U}}^{*}(\textbf{A} + \sigma_{i} \textbf{I}) \textbf{U})^{-1} \widetilde{\textbf{U}}^{*}$ is the projector onto $(\textbf{A} + \sigma_{i} \textbf{I}) \mathcal{U}$ along $\widetilde{\mathcal{U}}^{\perp}$.
\par
From this equation we see that computing $\textbf{y}_{j}^{(i)}$ is equivalent to applying the following projection method.
\\ \newline
\noindent\fbox{%
    \parbox{\textwidth}{
   Find $\textbf{x}_{j}^{(i)} \in \mathcal{W}_{j}$  as an approximate solution corresponding to shift $\sigma_{i}$ for the projected and shifted linear system
   \begin{equation} \label{projectedproblem}
    (\textbf{I} - \Phi^{(i)})   ( \textbf{A} + \sigma_{i} \textbf{I}) \textbf{x}^{(i)} = (\textbf{I} - \Phi^{(i)}) \textbf{r}_{0}^{(i)}
   \end{equation}
    such that $\textbf{r}_{j}^{(i)} = (\textbf{I} - \Phi^{(i)} ) (\textbf{r}_{0}^{(i)} - (\textbf{A} + \sigma_{i} \textbf{I}) \textbf{x}_{j}^{(i)}) \perp \widetilde{\mathcal{W}}_{j} $.
    }
}
\newline \\
Once we have computed $\textbf{y}_{j}^{(i)}$ we can incorporate $\textbf{z}_{j}^{(i)}$ implicitly into the full solution approximation at step $j$ for shift $i$ to get
\begin{equation}\label{eq:general_solutionapproximation}
    \textbf{x}_{j}^{(i)} = \textbf{x}_{0}^{(i)} + \textbf{W}_{j} \textbf{y}_{j}^{(i)} + \textbf{U}(\widetilde{\textbf{U}}^{*}(\textbf{A}+\sigma_{i} \textbf{I}) \textbf{U})^{-1} \widetilde{\textbf{U}}^{*}(\textbf{r}_{0}^{(i)} - (\textbf{A} + \sigma_{i} \textbf{I}) \textbf{W}_{j} \textbf{y}_{j}^{(i)}).
\end{equation}

In this paper we propose to use the framework to derive recycled shifted block Krylov subspace methods by choosing $\mathcal{W}_{j}$ to be an appropriate block Krylov subspace in addition to making appropriate choices of the other subspaces $\widetilde{\mathcal{W}}, \widetilde{\mathcal{U}}$ and $\mathcal{U}$. A practical method then solves the problem for $\textbf{y}_{j}^{(i)}$ with these subspaces and computes the updated solution approximation \cref{eq:general_solutionapproximation} and corresponding residual. We now investigate appropriate choices of $\mathcal{W}_{j}$.
\subsection{Using the projected formulation}
We have just shown how the framework reduces to solving a projected problem \cref{eq:shiftedequationfory} for every shift $i$. The original formulation of the framework (see e.g  \cite{gaul2014recycling, MR3054589, MR2995778, MR3284571, MR4175151} ) then proposes to solve this projected problem over the appropriately projected Krylov subspace. While this does not give any issues for problems of the form \cref{eq:linearsystem} for which the framework was derived, problems arise with our extension of the framework to shifted systems of the form \cref{eq:no_superscript_shifted_system}. This is because the corresponding projected block Krylov subspace is now given by
\[
\mathcal{K}_{j}((\textbf{I} - \Phi^{(i)}) \mathcal{T}, \widehat{\textbf{R}})
\]
where $\mathcal{T}$ is the Sylvester operator corresponding to the shifted system \cref{eq:no_superscript_shifted_system} and the columns of $\widehat{\textbf{R}}$ are given by 
\[
\widehat{\textbf{r}}^{(i)} = \textbf{r}_{0}^{(i)} - (\textbf{A} + \sigma_{i}\textbf{I}) \textbf{U} (\widetilde{\textbf{U}}^{*}(\textbf{A} + \sigma_{i}\textbf{I}) \textbf{U})^{-1} \widetilde{\textbf{U}}^{*} \textbf{r}_{0}^{(i)}.
\]

This introduces two main problems with the use of this Krylov subspace which renders methods based on projected Krylov subspaces impractical.  Firstly, projected block Krylov subspaces are no longer shift invariant and in general \[
     \mathcal{K}_{j}((\textbf{I} - \Phi^{(i)}) \textbf{A}, \widehat{\textbf{R}}) \neq \mathcal{K}_{j}((\textbf{I} - \Phi^{(i)}) (\textbf{A} + \sigma_{i} \textbf{I}), \widehat{\textbf{R}}).
\]
This is one of the main difficulties one encounters when attempting to develop a subspace recycling method for the solution of shifted systems \cite{MR3530183, MR3212178}.
In addition, we also generally have that
\[
   \mathcal{K}_{j}((\textbf{I} - \Phi^{(i)}) \mathcal{T}, \widehat{\textbf{R}}) \neq  \mathcal{K}_{j}((\textbf{I} - \Phi^{(i)}) \textbf{A}, \widehat{\textbf{R}}).
\]
The practical issue here is that we cannot build a separate block Krylov subspace for each shift since this will require excessive storage and computation and violate our assumption that the same subspace $\mathcal{W}_{j}$ is used for every shift.

\subsection{rsbGMRES}
Although not derived from the framework, the rsbGMRES algorithm overcomes the difficulties from the previous section by using the same projected Krylov subspace for each shift $\mathcal{K}_{j}((\textbf{I} - \textbf{C}\textbf{C}^{*}), \widehat{\textbf{R}})$ where $\textbf{C} = \textbf{A}\textbf{U}$ is taken to have orthonormal columns. The use of this projector was motivated by the GCRO-DR algorithm \cite{MR2272183} for systems of the form \cref{eq:linearsystem}. Using this same projector for each shift allows for a practical implementation of the rsbGMRES algorithm, but it cannot fit within the framework since it is not using the appropriate projectors for every shift. The rsbGMRES algorithm also relies on the following invariance 
\[
    \mathcal{K}_{j}((\textbf{I} - \textbf{C} \textbf{C}^{*}) \textbf{A},\widehat{\textbf{R}}) =   \mathcal{K}_{j}((\textbf{I} - \textbf{C} \textbf{C}^{*}) \mathcal{T},\widehat{\textbf{R}}),
\]
which holds true only when $\mbox{Range}(\widehat{\textbf{R}}) \perp \mathcal{C}$ where $\mathcal{C}$ is the space spanned by the columns of $\textbf{C}$ \cite[Proposition 4.1]{MR3530183}. It thus requires at the beginning of each cycle to project all residuals into the space $\mathcal{C}^{\perp}$, and we note the residual projections used in rsbGMRES do not correspond to residuals projections from the framework using appropriate GMRES projections. The rsbGMRES algorithm then works by deriving a shift dependent minimization problem and using a shifted augmented Arnoldi relation to develop a suitable problem on a small subspace. We will not give any more details on the rsbGMRES algorithm but will refer the reader to \cite{MR3530183} as well as the rsbGMRES MATLAB implementation \cite{soodhalterrsbGMREScode}.
\section{Unprojected formulation of the framework allows for a practical recycled shifted block implementation for shifted systems}\label{sec:unprojrsbFOMandGMRES}

We propose to overcome the issues discussed in the previous section using the observation made in \cite{burke2022augmented} where it was shown that it is not necessary to view the framework in terms of solving a projected problem. This is because the  
projected problem (\ref{projectedproblem}) has an unprojected equivalent given by
\\ \newline
\noindent\fbox{%
    \parbox{\textwidth}{
   Find $\textbf{x}_{j}^{(i)} \in \mathcal{W}_{j}$  as an approximate solution corresponding to shift $\sigma_{i}$ for the shifted linear system
   \begin{equation}
     ( \textbf{A} + \sigma_{i} \textbf{I}) \textbf{x}^{(i)} =  \textbf{r}_{0}^{(i)}
   \end{equation}
    such that $\textbf{r}_{j}^{(i)} =  \textbf{r}_{0}^{(i)} - (\textbf{A} + \sigma \textbf{I}) \textbf{x}_{j}^{(i)} \perp (\textbf{I} - \Phi^{(i)})^{*} \widetilde{\mathcal{W}}_{j} $.
    }
}
\newline \\

The proof of this equivalence is a trivial extension of the proof shown in \cite{burke2022augmented} to the shifted case. 
This version of the problem yields a natural choice of $\mathcal{W}_{j}$ given by the unprojected block Krylov subspace $\mathcal{K}_{j}(\textbf{A},\textbf{R}_{0})$. Since this is an unprojected Krylov subspace we can avoid the difficulties discussed in the previous section and furthermore we can exploit the same invariance \cref{eq:sylvester_invariance} used by sbFOM and sbGMRES without artificially imposing an invariance by performing an additional projection of the residuals.

The shift dependent projector $(\textbf{I} - \Phi^{(i)})$ is now built into the constraint space and appears in the equation one must solve for the correction  $\textbf{y}_{j}^{(i)}$ for shift $i$ and is given by
\begin{equation}\label{eq:shiftedequationfory}
    (\widetilde{\textbf{W}}_{j}^{*}(\textbf{I}-\Phi^{(i)})( \textbf{A} + \sigma_{i} \textbf{I})\textbf{W}_{j}) \textbf{y}_{j}^{(i)} = \widetilde{\textbf{W}}^{*}_{j} (\textbf{I}-\Phi^{(i)}) \textbf{r}_{0}^{(i)}.
\end{equation}
This is solved \textit{after} building the unprojected block Krylov subspace and this equation can easily be solved for each shift separately without large computation or storage efforts. The solution approximation can then be updated as normal using equation \cref{eq:general_solutionapproximation}. 

In the next subsection we introduce the unprojected rsbFOM and unprojected rsbGMRES methods simply by taking appropriate choices of search and constraint spaces in the above formulation. One of the main advantages of these methods is that they reduce the amount of times we must apply a projector at each iteration. This is because an unprojected method completely avoids projections in the $j$ iterations of the Arnoldi process. In the next subsections it is clear to see our algorithms require the construction and application of a shift dependent projector $4 s$ times per cycle. Each construction of these projectors requires computing the inverse of a $k \times k$ matrix. 
This is in contrast to the rsbGMRES algorithm which requires construction and application of a fixed projector $j$ times, in addition to construction and application of a shift dependent projector $2 s$ times. The construction of this shift dependent projector also requires the inverse of a $k \times k$ matrix.
\subsection{unproj rsbFOM}
We will now derive an unprojected rsbFOM algorithm using the unprojected formulation of the framework. This involves simply selecting the appropriate choices of $\widetilde{\mathcal{W}}_{j}, \widetilde{\mathcal{U}}$ and $\mathcal{U}$ for a FOM method. From the previous discussion we know we must choose $\mathcal{W}_{j} = \mathcal{K}_{j}(\textbf{A},\textbf{R}_{0})$.  
For a FOM algorithm we choose $\widetilde{\mathcal{W}}_{j} = \mathcal{W}_{j}$ and $\widetilde{\mathcal{U}} = \mathcal{U}$. 
Similar to other recycle methods such as GCRO-DR \cite{MR2272183}, we will will impose that $\textbf{C} = \textbf{A} \textbf{U}$ yielding the projector
\[
\Phi^{(i)} = (\textbf{C}+\sigma_{i} \textbf{U})(\textbf{U}^{*} \textbf{C} + \sigma_{i} \textbf{U}^{*} \textbf{U})^{-1} \textbf{U}^{*}.
\]
We will not assume orthonormality of any of the columns of $\textbf{U}$ or $\textbf{C}$ in any of the methods developed in this paper. One may wish to impose orthonormality for algorithmic simplifications.
The problem one must solve for $\textbf{y}_{j}^{(i)}$ \cref{eq:shiftedequationfory} 
is now given by
\[
    (\textbf{W}_{j}^{*}(\textbf{I}-\Phi^{(i)})( \textbf{A} + \sigma_{i} \textbf{I})\textbf{W}_{j}) \textbf{y}_{j}^{(i)} = \textbf{W}^{*}_{j} (\textbf{I}-\Phi^{(i)}) \textbf{r}_{0}^{(i)}.
\]

We can use the following relation
\[
      \textbf{r}_{0}^{(i)} = \textbf{W}_{j} \textbf{E}_{j} \textbf{S}_{0} \textbf{e}_{i},
\]
in combination with the shifted block Arnoldi relation \cref{eq:block_Arnoldi} to yield the following cheap problem to be solved for each $\textbf{y}_{j}^{(i)}, i=1,..,s$,
\begin{equation*}
  \begin{aligned}
    \bigg( (\textbf{H}_{j}+\sigma_{i} \textbf{I}) - (\textbf{W}_{j}^{*} \textbf{C} + \sigma_{i} \textbf{W}_{j}^{*} \textbf{U})(\textbf{U}^{*}\textbf{C} + \sigma_{i} \textbf{U}^{*}\textbf{U})^{-1} \textbf{U}^{*} \textbf{W}_{j+1}(\overline{\textbf{H}}_{j} + \sigma_{i} \overline{\textbf{I}})   \bigg) \textbf{y}_{j}^{(i)} \\ = \textbf{E}_{j} \textbf{S}_{0}\textbf{e}_{i} - (\textbf{W}_{j}^{*} \textbf{C} + \sigma_{i} \textbf{W}_{j}^{*} \textbf{U})(\textbf{U}^{*}\textbf{C} + \sigma_{i} \textbf{U}^{*} \textbf{U})^{-1} \textbf{U}^{*} \textbf{r}_{0}^{(i)}.
    \end{aligned}
\end{equation*}

At step $j$ once we have solved the shifted problem for each $\textbf{y}_{j}^{(i)}$, the solution approximation \cref{eq:general_solutionapproximation} can be computed as
\[
    \textbf{x}_{j}^{(i)} = \textbf{x}_{0}^{(i)} + \textbf{W}_{j} \textbf{y}_{j}^{(i)} + \textbf{U}(\textbf{U}^{*}\textbf{C} + \sigma_{i} \textbf{U}^{*}\textbf{U})^{-1}\textbf{U}^{*}(\textbf{r}_{0}^{(i)} - \textbf{W}_{j+1}(\sigma_{i} \overline{\textbf{I}} + \overline{\textbf{H}}_{j}) \textbf{y}_{j}^{(i)}).
\]
The corresponding residual at step $j$ is given by
\[
    \textbf{r}_{j}^{(i)} = \textbf{r}_{0}^{(i)} - \textbf{W}_{j+1}(\sigma_{i} \overline{\textbf{I}} + \overline{\textbf{H}}_{j}) \textbf{y}_{j}^{(i)} - (\textbf{C}+\sigma_{i} \textbf{U})(\textbf{U}^{*}\textbf{C}+\sigma_{i} \textbf{U}^{*}\textbf{U})^{-1} \textbf{U}^{*}(\textbf{r}_{0}^{(i)} - \textbf{W}_{j+1}(\sigma_{i} \overline{\textbf{I}} + \overline{\textbf{H}}_{j})\textbf{y}_{j}^{(i)}).
\]
The unprojected rsbFOM algorithm is outlined in \Cref{alg:unprojrsbFOM}.

\begin{algorithm}[H]
\caption{unprojected recycled shifted block FOM (unproj  rsbFOM)
\label{alg:unprojrsbFOM}} 
\begin{algorithmic}[1]
\STATE{\textbf{Input:} $\textbf{A} \in \mathbb{C}^{n \times n}$,  $\textbf{b}^{(1)}, \textbf{b}^{(2)}, ..., \textbf{b}^{(s)} \in \mathbb{C}^{n}$, $\textbf{x}_{0}^{(1)}, \textbf{x}_{0}^{(2)},...,\textbf{x}_{0}^{(s)} \in \mathbb{C}^{n}$, shifts $\sigma_{1},..,\sigma_{s} \in \mathbb{C}$, $\textbf{U} \in \mathbb{C}^{n \times k} $ whose columns span the augmenting recycled subspace $\mathcal{U}$,  $\textbf{C} = \textbf{A} \textbf{U}$, tolerance $\epsilon$, Arnoldi cycle length $j$.}
\FOR{$i=1,2,..,s$}
  \STATE{$\textbf{r}^{(i)}_{0} = \textbf{b}^{(i)} - (\textbf{A}+\sigma_{i}\textbf{I}) \textbf{x}^{(i)}_{0}$}
\ENDFOR
\IF{$\textbf{r}_{i} = \beta \textbf{r}_{k}$ for some $i \neq k$ and $\beta \in \mathbb{C}$}
\STATE{Render residuals non-colinear.}
\ENDIF
\STATE{Set $\textbf{R} \leftarrow \begin{bmatrix} \textbf{r}_{0}^{(1)} & \textbf{r}_{0}^{(2)} & ... & \textbf{r}_{0}^{(s)} \end{bmatrix}$}
\WHILE{$\|\textbf{R}\| > \epsilon$}
\STATE{Build a basis for the unprojected block Krylov subspace $\mathcal{K}_{j}(\textbf{A},\textbf{R})$ via the block Arnoldi algorithm generating $\textbf{W}_{j+1} \in \mathbb{C}^{n \times (j+1)s}$ and $\overline{\textbf{H}}_{j} \in \mathbb{C}^{(j+1)s\times js}$.}
\FOR{$i = 1,.., s$}
\STATE{$\Phi \leftarrow (\textbf{C}+\sigma_{i} \textbf{U})(\textbf{U}^{*}\textbf{C} + \sigma_{i} \textbf{U}^{*}\textbf{U})^{-1} \textbf{U}^{*}$}
\STATE{$\textbf{M} \leftarrow \textbf{W}_{j}^{*} (\textbf{I} - \Phi) \textbf{W}_{j+1}(\sigma_{i} \overline{\textbf{I}} + \overline{\textbf{H}}_{j}) $}
\STATE{$\textbf{g} \leftarrow \textbf{W}_{j}^{*}(\textbf{I} - \Phi) \textbf{r}_{0}^{(i)}$}
\STATE{Solve $ \textbf{M} \textbf{y} = \textbf{g}$ for $\textbf{y}_{j}$}
\STATE{$\textbf{s} \leftarrow \textbf{r}_{0}^{(i)} - \textbf{W}_{j+1}(\sigma_{i} \overline{\textbf{I}} + \overline{\textbf{H}}_{j}) \textbf{y}$}
\STATE{ $\textbf{x}_{j}^{(i)} \leftarrow \textbf{x}_{0}^{(i)} + \textbf{W}_{j} \textbf{y} + \textbf{U}(\textbf{U}^{*}\textbf{C} + \sigma_{i} \textbf{U}^{*}\textbf{U})^{-1}\textbf{U}^{*} \textbf{s}$}
\STATE{$\textbf{r}_{j}^{(i)} \leftarrow \textbf{s}
- \Phi \textbf{s}$}
\ENDFOR
\STATE{$\textbf{R} \leftarrow \begin{bmatrix} \textbf{r}_{j}^{(1)} & \textbf{r}_{j}^{(2)} & ... & \textbf{r}_{j}^{(s)}\end{bmatrix}$}
\STATE{Update $\textbf{U}$ and $\textbf{C}$ such that $\textbf{C} = \textbf{A} \textbf{U}$ to use for next matrix $\textbf{A}$}.
\ENDWHILE
\end{algorithmic}
\end{algorithm}

\subsection{unproj rsbGMRES}
Similar ideas as discussed in the development of unprojected rsbFOM also hold true in the development of unprojected rsbGMRES except this time we take $\mathcal{W}_{j} = \mathcal{K}_{j}(\textbf{A},\textbf{R}_{0})$, $\widetilde{\mathcal{W}} = (\textbf{A} + \sigma_{i}\textbf{I}) \mathcal{W}_{j}$  and  $\widetilde{\mathcal{U}} = (\textbf{A}+\sigma \textbf{I}) \mathcal{U}$. 
We solve the problem \cref{eq:shiftedequationfory} for $\textbf{y}_{j}^{(i)}$ with these subspace choices and the GMRES projector
\[
\Phi^{(i)} = (\textbf{C} + \sigma_{i}\textbf{U})((\textbf{C} + \sigma_{i}\textbf{U})^{*}(\textbf{C} + \sigma_{i}\textbf{U}))^{-1}(\textbf{C} + \sigma_{i}\textbf{U})^{*}.
\]

At step $j$ the solution approximation for shift $i$ is 
\begin{align*}
    \textbf{x}_{j}^{(i)} = \textbf{x}_{0}^{(i)} + \textbf{W}_{j}\textbf{y}_{j}^{(i)} + \textbf{U}((\textbf{C}+\sigma_{i}\textbf{U})^{*}(\textbf{C}+\sigma_{i}\textbf{U}))^{-1}(\textbf{C} + \sigma_{i} \textbf{U})^{*}(\textbf{r}_{0}^{(i)} - \textbf{W}_{j+1}(\overline{\textbf{H}}_{j} + \sigma_{i}\overline{\textbf{I}}) \textbf{y}_{j}^{(i)}),
\end{align*}
and the corresponding residual is 
\[
    \textbf{r}_{j}^{(i)} = \textbf{r}_{0}^{(i)} - \textbf{W}_{j+1} (\sigma_{i}\overline{\textbf{I}} + \overline{\textbf{H}}_{j})\textbf{y}_{j}^{(i)} - \Phi^{(i)}(\textbf{r}_{0}^{(i)} - \textbf{W}_{j+1} (\sigma_{i}\overline{\textbf{I}} + \overline{\textbf{H}}_{j})\textbf{y}_{j}^{(i)}).
\]

From the above projections we see that one could potentially impose $\textbf{C}$ to have orthonormal columns to avoid computing $\textbf{C}^{*}\textbf{C}$ each time the projector is constructed. The unprojected rsbGMRES algorithm is outlined below in \cref{alg:unprojrsbGMRES}.

\begin{algorithm}[H]
\caption{unprojected recycled shifted block GMRES (unproj rsbGMRES)
\label{alg:unprojrsbGMRES}} 
\begin{algorithmic}[1]
\STATE{\textbf{Input:} $\textbf{A} \in \mathbb{C}^{n \times n}$,  $\textbf{b}^{(1)}, \textbf{b}^{(2)}, ..., \textbf{b}^{(s)} \in \mathbb{C}^{n}$, $\textbf{x}_{0}^{(1)}, \textbf{x}_{0}^{(2)},...,\textbf{x}_{0}^{(s)} \in \mathbb{C}^{n}$, shifts $\sigma_{1},..,\sigma_{s} \in \mathbb{C}$, $\textbf{U} \in \mathbb{C}^{n \times k} $ whose columns span the augmenting recycled subspace $\mathcal{U}$,  $\textbf{C} = \textbf{A} \textbf{U}$, tolerance $\epsilon$, Arnoldi cycle length $j$.}
\FOR{$i=1,2,..,s$}
  \STATE{$\textbf{r}^{(i)}_{0} = \textbf{b}^{(i)} - (\textbf{A}+\sigma_{i}\textbf{I}) \textbf{x}^{(i)}_{0}$}
\ENDFOR
\IF{$\textbf{r}_{i} = \beta \textbf{r}_{k}$ for some $i \neq k$ and $\beta \in \mathbb{C}$}
\STATE{Render residuals non-collinear.}
\ENDIF
\STATE{Set $\textbf{R} \leftarrow \begin{bmatrix} \textbf{r}_{0}^{(1)} & \textbf{r}_{0}^{(2)} & ... & \textbf{r}_{0}^{(s)} \end{bmatrix}$}
\WHILE{$\|\textbf{R}\| > \epsilon$}
\STATE{Build a basis for the unprojected block Krylov subspace $\mathcal{K}_{j}(\textbf{A},\textbf{R})$ via the block Arnoldi algorithm generating $\textbf{W}_{j+1} \in \mathbb{C}^{n \times (j+1)s}$ and $\overline{\textbf{H}}_{j} \in \mathbb{C}^{(j+1)s\times j s}$.}
\FOR{$i = 1,.., s$}
\STATE{$\Phi \leftarrow  (\textbf{C} + \sigma_{i}\textbf{U})((\textbf{C}+\sigma_{i}\textbf{U})^{*}(\textbf{C} + \sigma_{i}\textbf{U}))^{-1}(\textbf{C}+\sigma_{i} \textbf{U})^{*}$}
\STATE{$\textbf{M} \leftarrow  (\textbf{W}_{j+1}(\sigma_{i}\overline{\textbf{I}} + \overline{\textbf{H}}_{j}))^{*}(\textbf{I}-\Phi) \textbf{W}_{j+1}( \sigma_{i} \overline{\textbf{I}} + \overline{\textbf{H}}_{j})$}
\STATE{$\textbf{g} \leftarrow (\textbf{W}_{j+1}(\sigma_{i}\overline{\textbf{I}} + \overline{\textbf{H}}_{j}))^{*} (\textbf{I}-\Phi) \textbf{r}_{0}^{(i)}$}
\STATE{Solve $ \textbf{M} \textbf{y} = \textbf{g}$ for $\textbf{y}$}
\STATE{$\textbf{s} \leftarrow \textbf{r}_{0}^{(i)}  - \textbf{W}_{j+1} (\overline{\textbf{H}}_{j} + \sigma_{i} \overline{\textbf{I}}) \textbf{y}$ }
\STATE{$\textbf{x}_{j}^{(i)} \leftarrow \textbf{x}_{0}^{(i)} + \textbf{W}_{j}\textbf{y} + \textbf{U}((\textbf{C}+\sigma_{i}\textbf{U})^{*}(\textbf{C}+\sigma_{i}\textbf{U}))^{-1}(\textbf{C} + \sigma_{i} \textbf{U})^{*} \textbf{s}$}
\STATE{$\textbf{r}_{j}^{(i)} \leftarrow \textbf{s} - \Phi \textbf{s} $}
\ENDFOR
\STATE{$\textbf{R} \leftarrow \begin{bmatrix} \textbf{r}_{j}^{(1)} & \textbf{r}_{j}^{(2)} & ... & \textbf{r}_{j}^{(s)}\end{bmatrix}$}
\STATE{Update $\textbf{U}$ and $\textbf{C}$ such that $\textbf{C} = \textbf{A} \textbf{U}$ to use for next matrix $\textbf{A}$}.
\ENDWHILE
\end{algorithmic}
\end{algorithm}

\section{Block harmonic Ritz recycling for unprojected recycled shifted block methods}\label{sec:harmritz}
In this section we derive a shift dependent harmonic Ritz extraction procedure for updating the recycling subspace $\mathcal{U}$ in unprojected rsbFOM and unprojected rsbGMRES. This gives the freedom to choose for which shift we construct the recycling subspace. For methods with shifted systems it is common to construct a recycle subspace from a Ritz or harmonic Ritz procedure corresponding to a shift of $0$. This was done in the rsbGMRES algorithm \cite{soodhalterrsbGMREScode, MR3530183} and in general the  recycling subspace works well enough for all shifts. However, as noted in \cite{MR3530183} for shifted systems it may not always be the case that a recycling subspace computed according to a base shift works well for all shifted systems and this is our motivation for including a shift dependent harmonic Ritz procedure for our methods.
One approach may then be to construct a recycle subspace for a new shift at each cycle of the algorithm.

Below we derive this harmonic Ritz procedure using the following notation
\[\widehat{\textbf{V}}_{j} = \begin{bmatrix}
      \textbf{U} & \textbf{W}_{j}
\end{bmatrix} \in \mathbb{C}^{n \times (k + s j)}, \widehat{\textbf{W}}_{j+1} = \begin{bmatrix} \textbf{C} & \textbf{W}_{j+1} \end{bmatrix} \in \mathbb{C}^{n \times (k + (j+1) s)}\]
and 
\[
\overline{\textbf{G}}_{j}^{(i)}  = \begin{bmatrix}
      \textbf{I} & 0 \\ 0 & \sigma_{i} \overline{\textbf{I}} + \overline{\textbf{H}}_{j}
\end{bmatrix} \in \mathbb{C}^{(k + (j+1) s) \times (k + j s)}.
\]
\begin{proposition}
\normalfont The harmonic Ritz problem for the matrix $\textbf{A} + \sigma_{i} \textbf{I}$ with respect to the augmented block Krylov subspace space $\mathcal{K}_{j}(\textbf{A}, \textbf{R}) + \mathcal{U}$ involves solving the following eigenproblem 
\begin{equation}\label{eq:general_harmritz}
     (\widehat{\textbf{W}}_{j+1}\overline{\textbf{G}}^{(i)}_{j} + \textbf{M}^{(i)})^{*} \widehat{\textbf{V}}_{j} \textbf{g}_{\ell} = \mu_{\ell} (\widehat{\textbf{W}}_{j+1}\overline{\textbf{G}}_{j}^{(i)} + \textbf{M}^{(i)})^{*} (\widehat{\textbf{W}}_{j+1}\overline{\textbf{G}}_{j}^{(i)} + \textbf{M}^{(i)}) \textbf{g}_{\ell},
\end{equation}
with $\mu_{\ell} \in \mathbb{C}$, $\textbf{g}_{\ell} \in \mathbb{C}^{js+k}$ and $\textbf{M}^{(i)} = \begin{bmatrix} \sigma_{i} \textbf{U} & 0 \end{bmatrix} \in \mathbb{C}^{n \times (k + j s)} $.
\end{proposition}
\begin{proof}
Combining the shifted block Arnoldi relation \cref{eq:block_shift_invariance} and $\textbf{C} = \textbf{A} \textbf{U}$ yields the following augmented Arnoldi relation
\begin{equation}\label{eq:augmentedArnoldi}
(\textbf{A} + \sigma_{i} \textbf{I}) \begin{bmatrix} \textbf{U} & \textbf{W}_{j} \end{bmatrix} = \begin{bmatrix} \textbf{C} & \textbf{W}_{j+1} \end{bmatrix} \begin{bmatrix}
      \textbf{I} & 0 \\
      0 & \sigma_{i} \overline{\textbf{I}} + \overline{\textbf{H}}_{j}
\end{bmatrix} + \begin{bmatrix}
      \sigma_{i} \textbf{U} & 0
\end{bmatrix}.
\end{equation}
The harmonic Ritz problem for the matrix $\textbf{A} + \sigma_{i} \textbf{I}$ with respect to the space $\mathcal{K}_{j}(\textbf{A},\textbf{R}) + \mathcal{U}$ requires one to solve the following problem
 \begin{align*}
    \mbox{Find} \hspace{0.5cm} (\textbf{y}_{\ell},\mu_{\ell}) \hspace{0.5cm} \mbox{such that} \hspace{0.5cm} (\textbf{A} + \sigma_{i} \textbf{I})^{-1}\textbf{y}_{\ell} - \mu_{\ell} \textbf{y}_{\ell} \perp (\textbf{A} + \sigma_{i} \textbf{I})(\mathcal{K}_{j}( \textbf{A}, \textbf{R}) + \mathcal{U}) \\ \mbox{with} \hspace{0.5cm} \textbf{y}_{\ell} \in (\textbf{A} + \sigma_{i} \textbf{I})(\mathcal{K}_{j}( \textbf{A}, \textbf{R}) + \mathcal{U}).
\end{align*}
Since $\textbf{y}_{\ell}$ is taken from the augmented block Krylov subspace $(\textbf{A} + \sigma_{i} \textbf{I})(\mathcal{K}_{j}( \textbf{A}, \textbf{R}) + \mathcal{U})$, it can be written as $\textbf{y}_{\ell} = (\textbf{A} + \sigma_{i} \textbf{I}) \widehat{\textbf{V}}_{j} \textbf{g}_{\ell}$ for some $\textbf{g}_{\ell} \in \mathbb{C}^{k + s j}$. 
Application of the orthogonality constraint yields an eigenproblem for $\textbf{g}_{\ell}$
\begin{align*}
    0 = ((\textbf{A}+\sigma_{i} \textbf{I}) \widehat{\textbf{V}}_{j})^{*}( (\textbf{A}+\sigma_{i} \textbf{I})^{-1}(\textbf{A} + \sigma_{i} \textbf{I}) \widehat{\textbf{V}}_{j} \textbf{g}_{\ell} - \mu_{\ell} (\textbf{A}+\sigma_{i} \textbf{I}) \widehat{\textbf{V}}_{j} \textbf{g}_{\ell})
    \\ =   (\widehat{\textbf{W}}_{j+1}\overline{\textbf{G}}_{j}^{(i)} + \textbf{M}^{(i)})^{*} (\widehat{\textbf{V}}_{j} \textbf{g}_{\ell} - \mu_{\ell} (\widehat{\textbf{W}}_{j+1}\overline{\textbf{G}}_{j}^{(i)} + \textbf{M}^{(i)}) \textbf{g}_{\ell})
    \\ =   (\widehat{\textbf{W}}_{j+1}\overline{\textbf{G}}_{j}^{(i)} + \textbf{M}^{(i)})^{*} \widehat{\textbf{V}}_{j} \textbf{g}_{\ell} - \mu_{\ell} (\widehat{\textbf{W}}_{j+1}\overline{\textbf{G}}_{j}^{(i)} + \textbf{M}^{(i)})^{*} (\widehat{\textbf{W}}_{j+1}\overline{\textbf{G}}^{(i)}_{j} + \textbf{M}^{(i)}) \textbf{g}_{\ell}.
\end{align*}
After solving this eigenproblem for each $\textbf{g}_{\ell}, \ell = 1,2,..,k$, we then compute
\[
\textbf{y}_{\ell} = (\textbf{A} + \sigma_{i} \textbf{I}) \widehat{\textbf{V}}_{j} \textbf{g}_{\ell}.
\]
\end{proof}
If we were interested in just computing $\mathcal{U}$ according to the shift $\sigma_{i} = 0$ then this problem reduces to the following eigenproblem
\begin{equation}\label{eq:special_harmritz}
     \overline{\textbf{G}}_{j}^{*} \widehat{\textbf{W}}_{j+1}^{*} \widehat{\textbf{V}}_{j} \textbf{g}_{\ell} = \mu_{\ell} \overline{\textbf{G}}_{j}^{*} \widehat{\textbf{W}}_{j+1}^{*} \widehat{\textbf{W}}_{j+1}\overline{\textbf{G}}_{j} \textbf{g}_{\ell},
\end{equation}
where $\overline{\textbf{G}}_{j}$ is simply $\overline{\textbf{G}}_{j}^{(i)}$ with $\sigma_{i} = 0$.

 In contrast to GCRO-DR type methods \cite{MR2272183}, our methods are built from unprojected Krylov subspaces and thus the columns of $\widehat{\textbf{W}}_{j+1}$ are not orthonormal and the product $\widehat{\textbf{W}}^{*}_{j+1} \widehat{\textbf{W}}_{j+1}$ still appears in the harmonic Ritz problem.

\section{Numerical Experiments}\label{sec:numerical_experiments}
In this Section we present our numerical experiments used to demonstrate the effectiveness of unprojected rsbFOM and unprojected rsbGMRES as both an augmented Krylov subspace method and a recycled Krylov subspace method. The augmentation subspaces we constructed in these experiments used the harmonic Ritz procedure we derived for shifted systems in \Cref{sec:harmritz}.

In the first experiment we test how our methods perform as augmented Krylov subspace methods. We do this by solving a single shifted system and varying the augmentation subspace dimension. We should expect that increasing the recycle subspace dimension leads to a faster rate of convergence for the iterative solver. In \cref{fig:unprojrsbFOM} and \Cref{fig:unprojrsbGMRES} we compare both unprojected rsbFOM and unprojected rsbGMRES to sbFOM and sbGMRES respectively. We see how our augmented methods always lead to a more rapid convergence in the $2$ norm of the block residual. The rate of convergence increases as we increase the augmentation subspace dimension. This is to be expected for any viable augmented Krylov subspace method.
\begin{figure}[H]
    \centering
    \includegraphics[width=8cm]{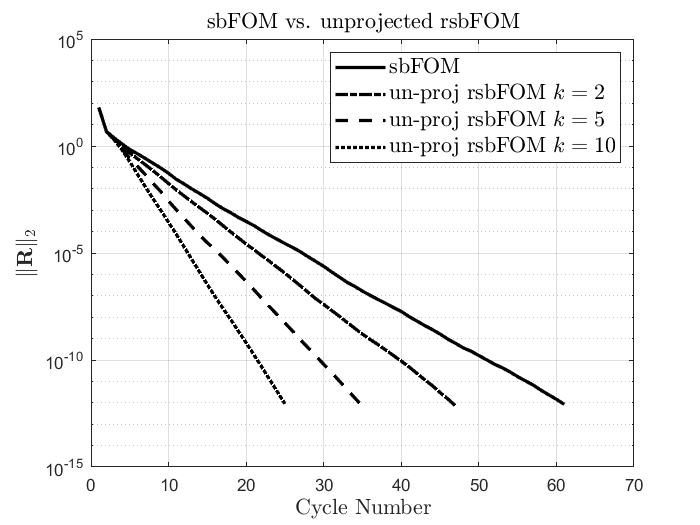}
    \caption{A comparison of sbFOM with unprojected rsbFOM (un-proj rsbFOM) for a Lattice QCD Wilson Dirac Matrix of size $3042\times 3042$ with $j = 10$, shifts = $0,1,2,3,4$ and $k=2,5,10$. The unprojected rsbFOM method was augmented with a recycle subspace coming from a previous system solve.}
    \label{fig:unprojrsbFOM}
\end{figure}
\begin{figure}[H]
    \centering
    \includegraphics[width=8cm]{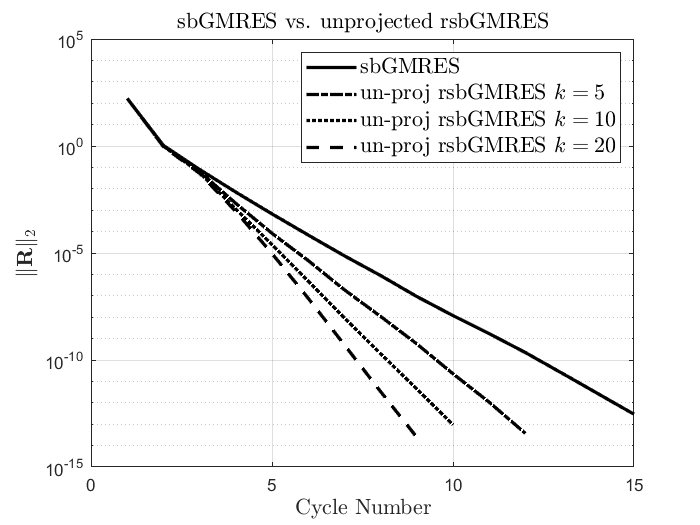}
    \caption{A comparison of sbGMRES with unprojected rsbGMRES for a Lattice QCD Wilson Dirac matrix of size $3042\times 3042$ with $j = 20$, shifts $= {0,0.01,0.02}$ and $k=5,10,20$. The unprojected rsbGMRES method was augmented with a recycle subspace coming from a previous system solve.}
    \label{fig:unprojrsbGMRES}
\end{figure}
We now demonstrate the effectiveness of unprojected rsbFOM and unprojected rsbGMRES as recycle methods. We show this by demonstrating  a significant reduction in matrix times block vector operations ( denoted as MAT-VECs) when solving a sequence of shifted linear systems of the form \cref{eq:shiftedsystem}. For these problems the matrices for each system is slowly changing, contains a new set of shifts, and a different block right hand side. We show results of our experiment for unprojected rsbGMRES, but note that similar behaviour can be observed for unprojected rsbFOM. We start off with a base matrix and introduce a perturbation parameter $\epsilon$ which determines the strength of how much the matrix changes. For the purpose of this experiment we perturb the matrix by some random matrix with the same sparsity pattern. We use a different set of shifts for each system in the sequence and a newly generated random block right hand side.  Our results are shown in \cref{fig:recycle_test}.

\begin{figure}[H]
    \centering
    \includegraphics[width=8cm]{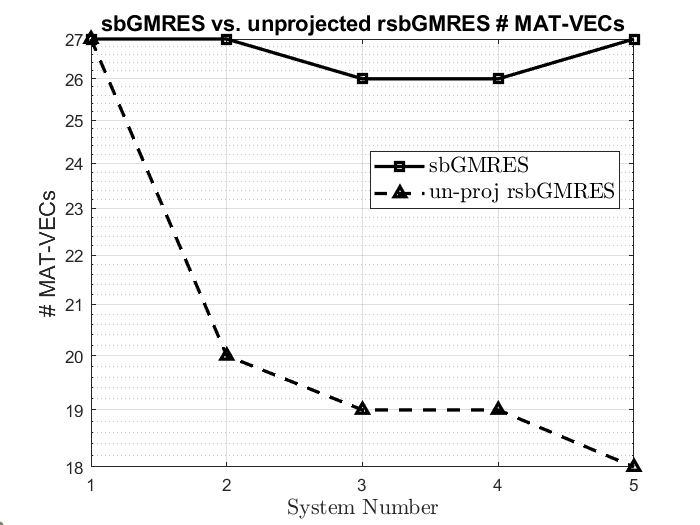}
    \caption{A comparison of the number of matrix times block vector operations (denoted MAT-VECs) required to meet a specified tolerance in sbGMRES and unprojected rsbGMRES when recycling between a family of shifted linear systems. We used a Lattice QCD Wilson Dirac matrix of size $3042\times 3042$ with $j = 15, k = 5$, and with $s$ shifts per family each separated by an increment of $0.0001$. We used a perturbation of strength $\epsilon = 0.01$ and new unrelated block right hand side for each system.}
    \label{fig:recycle_test}
\end{figure}
 Our experiment demonstrates a significant reduction in the matrix times  block vector operations required for solving each family in the sequence when using our new unprojected rsbGMRES algorithm. This makes the method a viable recycle method for problems of the form \cref{eq:shiftedsystem}.

In \Cref{fig:LQCDrsbGMRESvunprojrsbGMRES} and \Cref{fig:PoissonrsbGMRESvunprojrsbGMRES} we compare the convergence of rsbGMRES and unprojected rsbGMRES.  For a fair comparison we  ensure each method constructs a recycle subspace using its appropriate harmonic Ritz procedure. The methods have thus used different recycle subspaces in this experiment. In \Cref{sec:harmritz} we derived a shift dependent harmonic Ritz procedure for constructing a recycling subspace for our new unprojected methods using the appropriate augmented Arnoldi relation \cref{eq:augmentedArnoldi}. For the rsbGMRES algorithm this Arnoldi relation is the same except we now have that the matrix $\overline{\textbf{G}}_{j}^{(i)}$ is given by $\overline{\textbf{G}}_{j}^{(i)}  = \begin{bmatrix}
      \textbf{I} & \textbf{Z} \\ 0 & \sigma_{i} \overline{\textbf{I}} + \overline{\textbf{H}}_{j}
\end{bmatrix}$,  where $\textbf{Z} \in \mathbb{C}^{k \times j s}$ is a matrix storing the orthogonalization coefficients from doing the additional orthogonalizations required when building a basis for a projected Krylov subspace. It is easy to show that once this matrix is constructed the harmonic Ritz problem one must solve in rsbGMRES is the same as the one we derived for our unprojected methods \cref{eq:general_harmritz}. In the special case where one chooses to recycle only for the base shift $\sigma_{i} = 0$, then the simplified harmonic Ritz procedure \cref{eq:special_harmritz} can be used. Since rsbGMRES is based off a projected Krylov subspace this harmonic Ritz problem has orthonormal columns in $\widehat{\textbf{W}}_{j}$ and one can avoid computation of $\widehat{\textbf{W}}_{j}^{*} \widehat{\textbf{W}}_{j}$.
\begin{figure}[H]
    \centering
    \includegraphics[width=8cm]{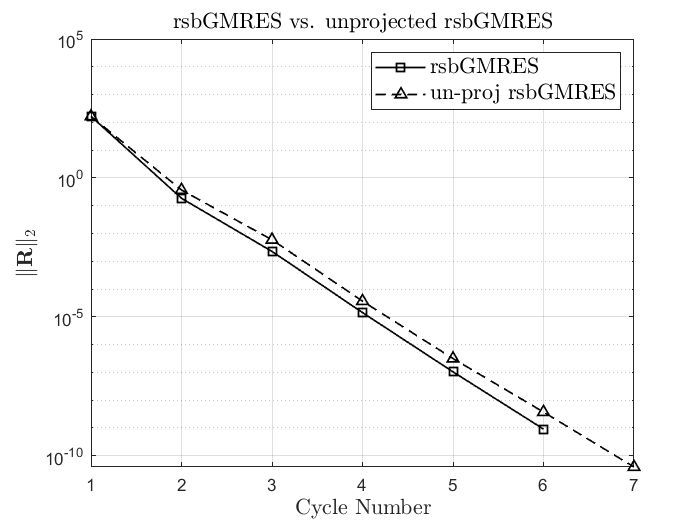}
    \caption{A comparison of the convergence behavior of rsbGMRES with unprojected rsbGMRES for a Lattice QCD Wilson Dirac matrix of size $3042\times 3042$ with $j = 30$, shifts $= {0,0.01,0.02}$ and $k=10$. Both codes solved the same system and were initially augmented with the same recycled subspace from a previous system.}
    \label{fig:LQCDrsbGMRESvunprojrsbGMRES}
\end{figure}

\begin{figure}[H]
    \centering
    \includegraphics[width=8cm]{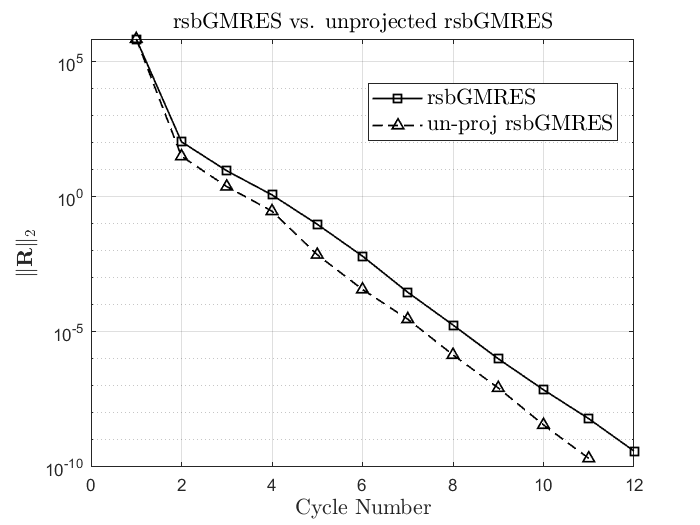}
    \caption{A comparison of the convergence behavior of rsbGMRES with unprojected rsbGMRES for a Poisson matrix of size $6400\times 6400$ with $j = 50$, shifts $= {0,1,2}$ and $k=20$. Both codes solved the same system and were initially augmented with the same recycled subspace from a previous system.}
    \label{fig:PoissonrsbGMRESvunprojrsbGMRES}
\end{figure}
  We see the convergence of unprojected rsbGMRES is very competitive with that of rsbGMRES. When we couple this with the computational advantages of  our unprojected methods discussed in \Cref{sec:unprojrsbFOMandGMRES}, it is clear our new methods represent a strong competitor for the rsbGMRES algorithm.
 
\section{Conclusions}\label{sec:conclusions}
In this paper we have presented a framework for deriving recycled shifted block Krylov subspace methods based on unprojected Krylov subspaces. Such methods are of interest since they allow us to avoid the collinearity restriction appearing in Krylov subspace methods for shifted systems. We have shown how methods based on unprojected Krylov subspaces allow one to avoid the difficulties associated with developing methods based on projected Krylov subspaces. We then introduced two new algorithms known as unprojected recycled shifted block FOM and GMRES and demonstrated their effectiveness as augmented and recycling methods. In addition, our methods reduce the amount of times projectors need to be applied in the iterations of the algorithm making them attractive in the High-Performance Computing setting. We have also derived a shift dependent harmonic Ritz procedure for shifted systems based on augmented block Krylov subspaces. This allows for one to construct a recycle subspace corresponding to a certain shift for unprojected recycled shifted block algorithms.

\section*{Acknowledgements}
I am grateful to the The Hamilton Scholars for generously funding this work. I would also like to thank my PhD advisor Kirk M. Soodhalter for useful discussions on the rsbGMRES algorithm.
\bibliographystyle{siamplain}
\bibliography{references}
\end{document}